\documentclass[reqno,a4paper,12pt]{amsart}
\usepackage{pdfsync}
\usepackage{a4wide}
\usepackage[all,cmtip]{xy}
\usepackage{amsfonts,stmaryrd,times}
\usepackage[mathcal]{eucal}
\usepackage{mathrsfs}
\usepackage{color}
\usepackage[pagebackref,colorlinks]{hyperref}
\usepackage{enumerate}

\theoremstyle{plain}
\newtheorem {Lem}{Lemma}[section]
\newtheorem {The}{Theorem}
\newtheorem {Prop}[Lem]{Proposition}

\theoremstyle{remark}
\newtheorem {Rem}[Lem]{Remark}
\newtheorem {Exp}[Lem]{Example}

\theoremstyle{definition}
\newtheorem {Def}[Lem]{Definition}
\newtheorem {defn}[Lem]{Definition}

\newcommand{\GL}{\operatorname{GL}}

\newcommand{\Hom}{\operatorname{Hom}}
\newcommand{\End}{\operatorname{End}}
\newcommand{\Auto}{\operatorname{Aut}}

\newcommand{\id}{\operatorname{id}}

\newif\ifcomm
\let\ifcomm\iffalse

\def\A{\operatorname{A}}

\def\E{\operatorname{E}}

\def\mC{\mathcal C}

\def\mS{\mathcal S}

\def\mP{\mathcal P}
\DeclareMathOperator{\mM}{Pgr}

\title{A matrix description for $K_1$ of  graded rings}

\author{Zuhong Zhang}
\address{Department of  Mathematics\\
 Beijing Institute of Technology\\
 Beijing\\ P. R. China}
\email{zuhong@gmail.com}

\thanks{The work was supported by the
joint Sino-Russian project 13-01-91150.}

\begin{document}

\begin{abstract}
The current paper is  dedicated to the study of the classical $K_1$ groups of graded rings.  Let $A$ be  a $\Gamma$ graded  ring with identity $1$, where the grading $\Gamma$ is an abelian group. We associate a category with suspension to the $\Gamma$ graded ring $A$. This allows us to construct the group valued functor  $K_1$ of graded rings. It will be denoted by $K_1^{gr}$. It is not only  an abelian group but also  a $\mathbb Z[\Gamma]$-module. From the construction, it follows that  there exists ``locally'' a matrix description of  $K_1^{gr}$ of graded rings.  The matrix description makes it possible to  compute $K_1^{gr}$ of various types of graded rings. The $K_1^{gr}$ satisfies the well known $K$-theory exact sequence 
$$
K_{1}^{gr}(A,I)\to K_1^{gr}(A)\to K_1^{gr}(A/I)
$$
for any graded ideal $I$ of $A$. The above is used to compute  $K^{gr}_1$ of cross products.
\end{abstract}

\maketitle
\tableofcontents

\section*{Introduction}
There has been a recent upturn in interest in graded $K$-theory, motivated by the classification problem for Leavitt path algebras \cite{RH1}, the representation theory of Hecke algebra \cite{AK}  and by  the graded $K$-theory itself \cite{RHTH}.

The classical group valued functor $K_1$ of rings  was  introduced  by Hyman Bass about half century ago. The matrix description of $K_1$ is  one of the most important tools available for computing $K_1$. However in the setting of graded rings, there is no obvious way of defining $K_1$ by matrices. In the current paper, we fill in this gap and give an explicit matrix description of $K_1$ of graded rings.

The strategy of our construction of $K_1$ of graded rings to build up   $K_1$ locally  from ``pieces'' and then ``glue'' all the  pieces by a direct limit.    Each piece, which we think of as a ``local'' $K_1$, has its own $K_1$ group  possessing a matrix description. Graded $K_1$ of graded rings is a direct limit of  the ``local'' $K_1$ groups, each of which has a matrices description.

More precisely, for any   $\Gamma$ graded ring $A$, the category $\mM A$ of graded projective $A$-modules is  a category with suspension (cf. 
Definition~\ref{def:suspension}). Let $S$ be a finite subfamily of elements of  the grading $\Gamma$. We construct a full subcategory $\mM A(S)$ of $\mM A$, which is a category with exact sequences (cf. Definition~\ref{def:K1_C}). The $K_1$ group of the subcategory $\mM A(S)$, denoted by $K_1(A)(S)$, is known by Bass' categorical definition of $K_1$ \cite[Definition~VII.1.4]{Bass}. These groups are our ``local'' $K_1$ groups. The $K_1(A)(S)$ groups together with the natural homomorphisms between them form a direct system of abelian groups, whose limit coincides with the $K_1$ of $\mM A$. On the other hand,  general arguments in linear algebra show that each $K_1(A)(S)$ possesses a matrix description. Thus graded $K_1$ of graded rings can be described ``locally'' by matrices. In general, the author does not know if there is a ``global'' matrix description of graded $K_1$ of graded rings.


The local matrix description of graded $K_1$ introduces a concrete tool for calculating graded $K_1$ groups. Graded Dedekind rings, graded local rings and  graded rings with finite graded stable rank (see \cite{RH2}) are promising cases for computing their graded $K_1$ groups. We are working towards this goal.


The rest of the paper is organized as follows. In Section~\ref{sec:cat},  we recall the notion of a category with suspension and  $K_1$ groups of  categories with suspension. In Section~\ref{s:1}, we briefly review some well-known facts of graded rings and graded 
modules, as well as  fix notation. We refer readers to \cite{Oy1,RH2} for more details. Then in Section~\ref{s:4}, we construct $K_1^{gr}$ from ``finite pieces''  and show that the new construction coincides with the categorical definition in Section~\ref{sec:cat}. Furthermore, the new definition allows us to obtain  an group action of the grading $\Gamma$  on $K_1^{gr}$.
In Section~\ref{matrix},  we  introduce  the notions of elementary and congruence  subgroups in the setting of graded rings, and study  the finite pieces of $K_1^{gr}$  using matrices.   In Section~\ref{s:5}, the well-known $K$-theory exact sequence is established  for $K_1^{gr}$. Finally, as an application, we show in Section~\ref{s:final} that if $A$ is a  cross product then $K_1^{gr}(A)\cong K_1(A_0)$ with trivial $\Gamma$ action.

\section{Category with suspension and its $K_1$}\label{sec:cat}
In this section,  we recall the categorical definition of $K_1$ groups due to Hyman Bass   in the setting of  categories with  suspension. 

\begin{Def}\label{def:K1_C}
A category $\mathcal C$ is called a {\em category with exact sequences} if $\mC$ is an abelian category with a full additive subcategory $\mathcal P$ satisfying the following additional conditions:
\begin{item}
\item[C1.] $\mathcal P$ is closed under extensions, i.e., if 
$$0\to  P_1\to P\to P_2\to 0$$
is an exact sequence in $\mathcal C$ and $P_1, P_2\in \mathcal P$ then $P\in \mathcal P$.
\item[C2.] $\mathcal P$ has a small skeleton (cf. \cite[Definition~3.1.1]{Rosen}).
\end{item}
\end{Def} 
\begin{Def}\label{def:suspension}
Let $\Gamma$ be an abelian group and  $\mC$  a category with exact sequences. Suppose that 
$$\mS_\Gamma=\{\mS_\lambda\text{ is an  exact auto-equivalence on }  \mC \mid \lambda\in\Gamma\},$$ 
and suppose that $\mS_\Gamma$ satisfies the following relations
\begin{itemize}
\item[1.] $\mS_0$ is the identity functor on $\mC$;
\item[2.] $\mS_\alpha\circ\mS_\beta=\mS_{\alpha+\beta}$ for every $\alpha,\beta\in \Gamma$.
\end{itemize}
The set $\mS_\Gamma$ forms an abelian group where multiplication is defined by composition of functors.
The  set $\mS_\Gamma$ is called a {\em suspension set} of $\mC$ and the functor  $\mS_\lambda\in \mS_\Gamma$  the {\em $\lambda$-suspension} in $\mS_\Gamma$.  By a {\em category with suspension}  we mean  a category with exact sequences  equipped with a suspension set. 
\end{Def}
\begin{Def}
Suppose that $\mC$ is a category with suspension, $\mS_\Gamma$ a suspension set on $\mC$, and  $\mP_0$ is a small skeleton of $\mC$. If  the restriction of $\mS_\lambda$ induces an auto-equivalence on $\mP_0$ for each $\lambda\in \Gamma$,  then $\mP_0$ is called a skeleton of $\mC$ with respect to the suspension set $\mS_\Gamma$.
\end{Def}
\begin{Lem}
Let $\mC$ be a category with a suspension $\mS_\Gamma$. Then there is a small skeleton $\mP_0$ of $\mC$ with respected to the suspension $\mS_\Gamma$.
\end{Lem}
\begin{proof}
Suppose that $\mP_1$ is a small skeleton of $\mC$. Then 
$$
\mP_0=\bigcup_{\lambda\in \Gamma} \mS_{\lambda}(\mP_1)
$$
which is a small skeleton of $\mC$.  Clearly the restriction of each $\mS_{\lambda}$ on $\mP_0$ is an equivalence. This proves the lemma.
\end{proof}
In the sequel, a small skeleton of $\mC$ will always mean that it respects the corresponding suspension of $\mC$.

\begin{Def}\label{def:K1}
Suppose that $\mathcal C$ is a category with exact sequences which has a small skeleton $\mathcal P_0$. Then $K_1(\mC)$ is defined to be the free abelian group on all pairs 
$(P,\sigma)$ such that $P\in \mathcal P_0$ and $\sigma\in \Auto( P)$,  subject to  the following relations:
\begin{item}
\item[(K1)] $[(P,\sigma)]+[(P,\tau)]=[(P,\sigma\circ\tau)]$;
\item[(K2)] $[(P_1,\sigma_1)]+[(P_2,\sigma_2)]=[(P,\sigma)]$, if there is a commutative diagram 
$$\xymatrix{0\ar[r]&P_1\ar[r]^i\ar[d]_{\sigma_1}&P\ar[r]^{\pi} \ar[d]_{\sigma}&P_2\ar[r]\ar[d]_{\sigma_2}&0\\
0\ar[r]&P_1\ar[r]^i&P\ar[r]^{\pi} &P_2\ar[r]&0
}
$$
in $\mC$ with exact rows.
\end{item}
\end{Def}
We recall the following well-known properties of $K_1(\mC)$.
\begin{Prop}\label{prop:K1}
Suppose that $\mathcal C$ is a category with exact sequences with a small skeleton $\mathcal P_0$. Then the group law, which we shall write additively, on $K_1(\mC)$ satisfies the following:
\begin{itemize}
\item[1.] $[(P_1, \id)]$ is the zero element in $K_1(\mC)$, for each $P_1\in \mP_0$;
\item[2.] the inverse of $[(P_1,\sigma)]$ in $K_1(\mC)$ is $[(P_1,\sigma^{-1})]$;
\item[3.] $[(P_1\oplus P_1, \sigma\oplus\sigma^{-1})]=[(P_1\oplus P_1,\id)]$ is the zero in $K_1(\mC)$.
\end{itemize}
\end{Prop}

Let $\mC$ be a category with  suspension $\mS_\Gamma$. By definition,  $\mS_{\Gamma}$ consists of exact auto-equivalences. Therefore each element $\mS_{\lambda}$ in $\mS_{\Gamma}$ preserves the identities (K1) and (K2) in Definition~\ref{def:K1}. This implies that each $\mS_{\lambda}$ induces an automorphism of $K_1(\mC)$. In  other words,  $\mS_{\Gamma}$ induces  a $\Gamma$-action  on $K_1(\mC)$.
\begin{Lem}\label{lem:sus}
Let $\mC$ be a category with suspension $\mS_\Gamma$. Then there is a group action of $\Gamma$ on $K_1(\mC)$ 
$$
\lambda \cdot [(P,\sigma)]=[(\mS_{\lambda}(P),\mS_{\lambda}(\sigma))],
$$
which is well-defined.
\end{Lem}
It follows immediately that $K_1(\mC)$ of a category $\mC$ with the suspension $\mS_\Gamma$ is not only an abelian group but also a $\mathbb Z[\Gamma]$-module.


\section{Graded rings and graded modules}\label{s:1}
In this section, we briefly review the notions of graded ring and graded module. 
 For details see \cite{Oy1} or \cite{RH2}.

Unless otherwise stated, all rings in the current paper are assumed to be associative with identity $1$, all ring homomorphisms are assumed to be identity preserving, and all the modules in sequel are left modules.

\begin{defn}
Let $A$ be a ring and  $\Gamma$ an abelian group. $A$ is called a  {\em $\Gamma$ graded ring}, if 
\begin{itemize}
\item $A=\oplus_{\lambda\in \Gamma}A_\lambda$, where  each $A_\lambda$ is an additive subgroup of $A$;
\item $A_\lambda A_\delta \subseteq A_{\lambda+\delta}$ for all $\lambda,\delta\in \Gamma$.
\end{itemize}
The none-zero elements in $A_\lambda$ are called {\em the homogeneous elements of level $\lambda$}.  We say  that a ring $A$ has  {\em trivial grading}, if $A$ is graded by the trivial group. 
\end{defn}

In the case that $A_\lambda A_\delta =A_{\lambda+\delta}$ for all $\lambda, \delta \in \Gamma$, we call $A$  a {\em  strongly $\Gamma$  graded ring}. If there is an invertible element in  $A_\lambda$ for each $\lambda\in \Gamma$, then we say that $A$ is a {\em cross product}.  A cross product is always a strongly graded ring, but in general the converse is not true.

Let $A$ be a $\Gamma$ graded ring. A two-sided ideal  $I$ of $A$ is called a {\em graded ideal } if 
$$
I=\bigoplus_{\lambda\in\Gamma }I\cap A_{\lambda}.
$$
We shall denote $I\cap A_{\lambda}$ by $I_\lambda$. It is called the {\em homogeneous component of level} $\lambda$ of $I$.
Then the quotient ring
$$
A/I= \bigoplus_{\lambda\in \Gamma} A_\lambda /I_{\lambda}.
$$
is again a $\Gamma$ graded ring.

Let $A$ and $B$ be two $\Gamma$ graded rings.  Then a {\em $\Gamma$ graded ring homomorphism} $\varphi: A\to B$ is a ring homomorphism  which respects to the $\Gamma$ grading, i.e.,
$$
\varphi (A_\lambda)\subseteq B_{\lambda}
$$
for all $\lambda \in \Gamma$. It is easy to check  the following facts:
\begin{itemize}
 \item[1.] The kernel of  a  graded homomorphism is a graded ideal;
 \item[2.] There is a canonical graded homomorphism  $A \to A/I$ for every graded ideal $I$ of $A$.
 \end{itemize}
 A graded homomorphism is called surjective, injective or bijective if it is so in the usual sense.  A bijective graded homomorphism is called a graded isomorphism and has a graded inverse.
 



 Let $A$ be a $\Gamma$ graded ring. An $A$-module is called a $\Gamma$ {\em graded $A$-module} (or simply a graded $A$-module when the grading is clear from  context), if $M$ is a directly sum   $M=\oplus_{\lambda\in \Gamma} M_\lambda$ of subgroups  and $A_\lambda M_{\delta}\subseteq M_{\lambda+\delta}$ for all $\lambda,\delta\in \Gamma$. We may define  graded submodule, graded quotient module in an obvious way.

Additionally, we have one further operation called {\em suspension} (or {\em shifting}) of a graded module. An {\em $\alpha$-suspended module} $M(\alpha)$ of $M$ is a $\Gamma$ graded module such that 
$$
M(\alpha)= \bigoplus_{\lambda\in \Gamma} M(\alpha)_\lambda,
$$
 where $M(\alpha)_\lambda =M_{\alpha+\lambda}$ and $\alpha\in \Gamma$.

Let $M$ and $N$ be graded $A$-modules and let $\delta\in \Gamma$.  A {\em graded module homomorphism of degree $\delta$} is an $A$-module homomorphism such that 
$\varphi(M_\lambda)\subseteq N_{\lambda+\delta}$ for each $\lambda\in \Gamma$. We denote the set of graded $A$-module homomorphisms of degree $\delta$ from $M$ to $N$   by $\Hom_{gr}(M,N)_\delta$ and the set of graded $A$-module endomorphism of degree $\delta$ on $M$ by  $\End_{gr}(M)_{\delta}$.

By a {\em graded module homomorphism}, we mean a graded module homomorphism of degree $0$. The set of graded $A$-module automorphism on $M$ is denoted by $\Auto_{gr}(M)$.

A  graded $A$-module $M$ is said to be free, if it has a free basis as an $A$-module consisting of homogeneous elements, equivalently  $M\cong_{gr}\oplus_{i\in I} A(\alpha_i)$ where $\{\alpha_i\mid i\in I\}$ is a family\footnote{By a family we mean a collection of elements whose elements can be repeated.} of elements of the grading $\Gamma$. When the basis happens to be finite, we denote  the graded free module $\oplus_{i=1}^{n} A(\alpha_i)$ by $ A(\alpha_1,\cdots,\alpha_n)$. 

$\End_{gr}(A(\alpha_1,\cdots,\alpha_n))_{\delta}$ possesses a  representation by  $n\times n$ matrices  (see \cite[\S 2.10]{Oy1}) :
$$
\End_{gr}(A(\alpha_1,\cdots,\alpha_n))_{\delta}
\cong \left(
\begin{array}{ccccc}
A_{\delta}& A_{\delta+\alpha_1-\alpha_2}& A_{\delta+\alpha_1-\alpha_3}&\cdots& A_{\delta+\alpha_1-\alpha_n}\\
A_{\delta+\alpha_2-\alpha_1}& A_{\delta}&A_{\delta+\alpha_2-\alpha_3}&\cdots& A_{\delta+\alpha_2-\alpha_n}\\
A_{\delta+\alpha_3-\alpha_1}& A_{\delta+\alpha_3-\alpha_2}&A_{\delta}&\cdots& A_{\delta+\alpha_3-\alpha_n}\\
\vdots&\vdots&\vdots&&\vdots\\
A_{\delta+\alpha_n-\alpha_1}& A_{\delta+\alpha_n-\alpha_2}&A_{\delta+\alpha_n-\alpha_3}&\cdots& A_{\delta}\\
\end{array}
\right)
$$
which is denoted by $M_n(A)(\alpha_1,\alpha_2,\cdots,\alpha_n)_{\delta}$. 

In particular, when the degree $\delta=0$, we have
$$
\End_{gr}(A(\alpha_1,\cdots,\alpha_n))_{0}
\cong  \left(
\begin{array}{ccccc}
A_{0}& A_{\alpha_1-\alpha_2}& A_{\alpha_1-\alpha_3}&\cdots& A_{\alpha_1-\alpha_n}\\
A_{\alpha_2-\alpha_1}& A_{0}&A_{\alpha_2-\alpha_3}&\cdots& A_{\alpha_2-\alpha_n}\\
A_{\alpha_3-\alpha_1}& A_{\alpha_3-\alpha_2}&A_{0}&\cdots& A_{\alpha_3-\alpha_n}\\
\vdots&\vdots&\vdots&&\vdots\\
A_{\alpha_n-\alpha_1}& A_{\alpha_n-\alpha_2}&A_{\alpha_n-\alpha_3}&\cdots& A_{0}\\
\end{array}
\right).
$$

Finitely generated projective modules play a crucial in algebraic $K$-theory. In the graded setting, a  graded  $A$-module $P$ is said to be projective, if 
for any diagram of grade homomorphisms with exact row
$$
\xymatrix{&P\ar[d]^i &\\
M\ar[r]^g&N\ar[r]&0,
}
$$
there is a grade $A$-module homomorphism $h: P\to M$ making the resulting diagram commutes.

\begin{Prop}\label{gr-projective}
Let $A$ be a $\Gamma$ graded ring and let $P$ be a graded projective $A$-module. Then the following conditions are equivalent:
\begin{itemize}
\item[1.] $P$ is a direct summand of a graded free $A$-module.
\item[2.] Every short exact sequence of graded $A$-module homomorphisms
$$
0\to M\to N\to P\to 0
$$
is split by  a graded $A$-module homomorphism.
\item[3. ] $Hom_A^{gr}(P, \underline{\phantom{A}} )$ is an exact functor.
\end{itemize}
\end{Prop}
\begin{proof}
See \cite{RH2}.
\end{proof}


\section{Category of finitely generated graded  projective modules and its $K_1$.}\label{K_1}\label{s:4}
In this section, we show that the category $\mM A$ of  finitely generated graded projective $A$-modules is a category with suspension.  By analogy with the usual procedures of defining $K_1$ (see \cite{Bass, Rosen, Magurn}),  we construct  stable graded $K_1^{gr}$ for graded rings in two different ways. One is the categorical definition of Section~\ref{sec:cat} and the other is an explicit definition using matrices. We shall show that these two definitions agree with each other.   

Let $A$ be a $\Gamma$ graded ring. The category of finitely generated graded projective $A$-modules $\mM A$  is defined as follows:
\begin{itemize}
\item the objects of $\mM {A} $ are finitely generated graded  projective $A$-modules.
\item the morphisms (arrows) are the graded homomorphisms between graded $A$-modules.
\end{itemize}
It is routine to show that the category $\mM A$ is a category with exact sequences. 

For each $\lambda\in \Gamma$, we define a map $\mS_{\lambda} : \mM A\to \mM A$  in the following way. For each object $M\in \mM A$, set $
\mS_{\lambda}(M)=M(\lambda),
$
and  for each morphism $\varphi: M\to N$ in $\mM A$, set $\mS_{\lambda}(\varphi): M(\lambda)\to N(\lambda)$ such that 
$$\mS_{\lambda}(\varphi)\big( M(\lambda)_\alpha\big) = \varphi (M_{\lambda+\alpha})$$
for every $\alpha\in \Gamma$.
It is easy to see that $\mS_{\lambda} : \mM A \to \mM A$ is an exact auto-equivalence.

Let $\mS_{\Gamma}=\{\mS_{\lambda}\mid\lambda\in \Gamma\}$. Then the category $\mM A$ is a category with the suspension $\mS_{\Gamma}$. This observation allow us to apply the functor $K_1$ to  the category $\mM A$ (see Section~\ref{sec:cat}).  We shall denote $K_1(\mM A )$ by $K_1^{gr}(A)$.

\begin{Lem}\label{functorial}
$K_1^{gr}(A)$ defines a  functor from the category of $\Gamma$ graded rings to the category of $\mathbb Z[\Gamma]$-modules.
\end{Lem}

Next, we will construct a family of full additive subcategories of $\mM A$, and calculate the operation of  the suspension functor on them.

Let $\{\alpha_1,\cdots, \alpha_n\}$ be a finite family of elements of  $\Gamma$.  The graded free module 
$$A\big(\underbrace{\alpha_1,\cdots,\alpha_n,\alpha_1,\cdots,\alpha_n,\cdots,\alpha_1,\cdots,\alpha_n}_{m \text{ copies of } \alpha_1,\cdots,\alpha_n}\big)$$
is denoted by $A(\alpha_1,\cdots,\alpha_n)^m$. 
Let $\mM A(\alpha_1,\cdots,\alpha_n)$ denote the full subcategory  of $\mM A$, whose objects are all  objects in $\mM A$ which are graded isomorphic to a direct summand of a graded module in $\{A(\alpha_1,\cdots,\alpha_n)^m\mid m\in\mathbb N\}$. Since $\mM A(\alpha_1,\cdots,\alpha_n)$ is a full additive subcategory of $\mM A$,  the group $K_1(\mM A(\alpha_1,\cdots,\alpha_n))$ can be defined by Bass' categorical approach \cite[Definition~VII.1.4]{Bass} and will be denoted by $K_1(A)(\alpha_1,\cdots,\alpha_n)$. 
\begin{Lem}
If $S\subseteq T$ are two finite families of $\Gamma$, then $\mM A(S)$ is a full subcategory of $\mM A(T)$. Furthermore, there is a natural group homomorphism $\varphi: K_1(A)(S)\to K_1(A)(T)$ induced by the inclusion functor from $\mM A(S)$ to $\mM A(T)$.
\end{Lem}
\begin{proof}
Since the inclusion functor is exact, it induces a map
\begin{eqnarray*}
\varphi: &K_1(A)(S)\to K_1(A)(T)&\\
&\varphi([(P,\sigma )])=[(P,\sigma)],
\end{eqnarray*}
which is clearly a group homomorphism.
\end{proof}

The set of all finite families $S$ of elements of $\Gamma$ forms a diagram of sets under natural inclusion of sets. This diagram is clearly directed. Thus the diagram of all $K_1(A)(S)$, where $S$ ranges over all finite families of elements of $\Gamma$ forms a directed digram of groups and thus has a direct limit $\varinjlim_{S} K_1(A)(S)$.

\begin{The}\label{the:1}
Let $A$ be a $\Gamma$ graded ring. Then there is a $\mathbb Z[\Gamma]$-module isomorphism
$$K_1^{gr}(A)\cong \varinjlim_{S} K_1(A)(S),$$
where the $\mathbb Z[\Gamma]$-action on $\varinjlim_{S} K_1(A)(S)$ is induced from the action on $K_1^{gr}(A)$.
\end{The}
\begin{proof}
For each finite family $S$ of elements of $\Gamma$, there is a group homomorphism $K_1(A)(S)\to K_1^{gr}(A)$ defined by sending every class $[(P,\sigma)]\in K_1(A)(S)$ to its corresponding class $[(P,\sigma)]$ in $K_1^{gr}(A)$. From the formalities of direct limits, we get a group homomorphism 
$$\varphi: \varinjlim_{S}K_1(A)(S)\to K_1^{gr}(A).$$ 
Each elements of $K_1^{gr}(A)$ is represented by an isomorphism classes $[(P,\sigma)]$, where the graded module $P$ is  finitely generated and projective. Hence the homomorphism  is clearly surjective. 

On the other hand, $K_1^{gr}(A)$ is the free group generated by isomorphism classes $[(P,\sigma)]$. Let $\displaystyle [(Q,\tau)]\in \varinjlim_{S}K_1(A)(S)$ such that $[(Q,\tau)]$ represents the identity  in $K_1^{gr}(A)$.   Then $ [(Q,\tau)]$ is generated by  finite terms of free generators in $K_1^{gr}(A)$. 
Hence there is a finite family $S$ of elements of $\Gamma$, such that all the given free generators of  $ [(Q,\tau)]$ sit in $K_1(A)(S)$.  Thus  $[(Q,\tau)]$ is already the identity element in $K_1(A)(S)$. Therefore it is the identity element in $\displaystyle\varinjlim_{S} K_1(A)(S)$. This shows that $\varphi$ is injective. Hence the whole theorem.
\end{proof}

\begin{Lem}
Let $\{\alpha_1,\cdots, \alpha_n\}$ be a finite family of  elements of $\Gamma$. For each $\lambda\in \Gamma$, the $\lambda$-suspension $\mS_\lambda$ induces a group isomorphism 
$$\varphi_\lambda: K_1(A)(\alpha_1,\cdots, \alpha_n) \to K_1(A)(\alpha_1+\lambda,\cdots, \alpha_n+\lambda).$$
\end{Lem}
\begin{proof}
Each $\lambda$-suspension $\mS_\lambda$  is an exact auto-equivalence on $\mM A$, and the image of $\mM A(\alpha_1,\cdots, \alpha_n)$ under functor $\mS_\lambda$ is $\mM A(\alpha_1+\lambda,\cdots, \alpha_n+\lambda)$. Therefore $\mS_\lambda$ induces a surjective group homomorphism 
$$\varphi_\lambda: K_1(A)(\alpha_1,\cdots, \alpha_n) \to K_1(A)(\alpha_1+\lambda,\cdots, \alpha_n+\lambda).$$ 
Conversely, $\mS_{-\lambda}$ induces a surjective group homomorphism
$$
\varphi_{-\lambda}:  K_1(A)(\alpha_1+\lambda,\cdots, \alpha_n+\lambda) \to K_1(A)(\alpha_1,\cdots, \alpha_n).
$$
By Definition~\ref{def:suspension}, the compositions  $\varphi_{\lambda}\circ\varphi_{-\lambda}$ and $\varphi_{-\lambda}\circ\varphi_{\lambda}$ are identity homomorphisms on  $K_1(A)(\alpha_1+\lambda,\cdots, \alpha_n+\lambda)$ and  $K_1(A)(\alpha_1,\cdots, \alpha_n)$, respectively. Hence $\varphi_{\lambda}$ is an isomorphism and this finishes the proof.
\end{proof}



\section{The matrix description of  $K_1^{gr}$.}\label{matrix}
We introduce the notions of elementary and congruence subgroups of stable general linear groups and   construct an explicit matrix description of graded  $K_1$ of graded rings.

Let $A$ be a $\Gamma$-graded ring with identity  as usual. 

\begin{Def} Let $A$ be a $\Gamma$ graded ring and let $\{\alpha_1,\alpha_2,\cdots,\alpha_n\}$ be a family of elements of $\Gamma$.  The set of all  invertible matrices in $M_n(A)(\alpha_1,\alpha_2,\cdots,\alpha_n)_0$ is called the {\em nonstable graded general linear group with parameter $(\alpha_1,\alpha_2,\cdots,\alpha_n)$} and will be  denoted by
$$
\GL_n(A)(\alpha_1,\alpha_2,\cdots,\alpha_n).
$$
\end{Def}

\begin{Def}
Let $A$ be a $\Gamma$ graded ring and let $\{\alpha_1,\alpha_2,\cdots,\alpha_n\}$ be a family of elements of $\Gamma$. The {\em  nonstable graded elementary subgroup of rank $n$} of $\GL_n(A)(\alpha_1,\cdots, \alpha_n)$ is defined as follows:
$$\E_n(A)(\alpha_1,\cdots, \alpha_n)=\langle e_{i,j}(r)\mid r\in A_{\alpha_i-\alpha_j}, i\ne j \rangle,$$
where $e_{i,j}(r)$ is an elementary matrix in the usual sense. Let $I$ be a graded two-sided ideal of $A$. Let $\varphi_I$ be the canonical map $A\to A/I$. It induces a group homomorphism, which is again denoted by $\varphi_I$, 
$$
\varphi_I: \GL_n(A)(\alpha_1,\cdots, \alpha_n)\to \GL_n(A/I)(\alpha_1,\cdots, \alpha_n).
$$ 
The kernel of $\varphi_I$ is called the {\em congruence subgroup $\GL_n(A,I)(\alpha_1,\cdots, \alpha_n)$ of level $I$} of 
$\GL_n(A)(\alpha_1,\cdots, \alpha_n)$.
The {\em graded elementary subgroup} $\E_n(I)(\alpha_1,\cdots, \alpha_n)$ of level $I$  is the  subgroup of  $\E_n(A)(\alpha_1,\cdots, \alpha_n)$ which is  generated by all $e_{i,j}(r)$ with $r\in I_{\alpha_j-\alpha_i}$. The minimal normal subgroup of $\E_n(A)(\alpha_1,\cdots, \alpha_n)$ containing $\E_n(I)(\alpha_1,\cdots, \alpha_n)$ is called the {\em graded relative elementary subgroup of level $I$} and is denoted by $\E_n(A,I)(\alpha_1,\cdots, \alpha_n)$.
\end{Def}
For every $m\in \mathbb N$, we embed  $\GL_{nm}(A)(\alpha_1,\cdots,\alpha_n)^m$ into $\GL_{n(m+1)}(A)(\alpha_1,\cdots,\alpha_n)^{m+1}$ by the  homomorphism $\varphi_{m}$ defined by
$$\begin{array}{lrcl} \varphi_{m}: &\GL_{nm}(A)(\alpha_1,\cdots,\alpha_n)^m&\to& \GL_{n(m+1)}(A)(\alpha_1,\cdots,\alpha_n)^{m+1}\\
&\varphi_{m}(g)&=&\left(
\begin{array}{cc}
g&0\\0&I_{n}
\end{array}
\right).
\end{array}
$$
The infinite union of the $\GL_{nm}(A)(\alpha_1,\cdots,\alpha_n)^m$ is denoted by $\GL(A)(\alpha_1,\cdots,\alpha_n)$. In the same way,  the infinite unions of $\GL_{nm}(A,I)(\alpha_1,\cdots, \alpha_n)^m$, $\E_{nm}(A)(\alpha_1,\cdots,\alpha_n)^m$ and $\E_{nm}(A,I)(\alpha_1,\cdots, \alpha_n)^{m}$ are denoted by $\GL(A,I)(\alpha_1,\cdots, \alpha_n)$, $\E(A)(\alpha_1,\cdots,\alpha_n)$ and \linebreak $\E(A,I)(\alpha_1,\cdots, \alpha_n)$ respectively.

In general, $\E_n(A)(\alpha_1,\cdots,\alpha_n)$  is not necessarily to be perfect, even when $n\ge 3$.

\begin{Exp}
Let $R$ be a commutative ring with a non-trivial ideal $I$. Then we construct a ring 
$A=(R,I)$ where addition is defined to be component-wise and the multiplication is defined as follows
$$
(r_1,a_1)(r_2,a_2)=(r_1r_2,r_1a_2+r_2a_1).
$$
The ring $A$ has a $\mathbb Z_3$ grading, namely
$$
A=(R,I)=(R,0)\oplus (0,I)\oplus (0,0),
$$
where $(R,0)$ is the homogeneous component of degree $0$, $(0,I)$ is the homogeneous component of degree $1$, and  the homogeneous component of degree $2$ is the trivial group.

Consider $\E_{3}(A)(0,1,2)$. A straight forward calculation shows that 
$$
[\E_{3}(A)(0,1,2),\E_{3}(A)(0,1,2)]\neq\E_{3}(A)(0,1,2).
$$

\end{Exp}

 However $\E_n(A)(\alpha_1,\cdots,\alpha_n)$ is indeed perfect, when $A$ happens to be a strongly graded ring with $n\ge 3$.
\begin{Lem}\label{perfect-1}
If $A$ is a strongly  $\Gamma$ graded ring, then $\E_n(A)(\alpha_1,\cdots,\alpha_n)$ is a perfect subgroup for any finite family $\{\alpha_1,\cdots,\alpha_n\}$ of elements of  $\Gamma$ and $n\ge 3$.
\end{Lem}
\begin{proof}
Suppose that $e_{i,j}(r)$ is a typical generator of $ \E_n(A)(\alpha_1,\cdots,\alpha_n)$ with $r\in A_{\alpha_i-\alpha_j}$. Since $A$ is strongly graded, 
$$
A_{\alpha_i -\alpha_k}A_{\alpha_k-\alpha_j}=A_{\alpha_i-\alpha_j}
$$
for any $0\le k\le n$. By the assumption $n\ge 3$,  we may choose $k\ne i , j$. There is a decomposition 
$$ 
r= s_1t_1+s_2t_2+\cdots+s_mt_m
$$
with $s_l\in A_{\alpha_i-\alpha_k}$ and $t_l\in A_{\alpha_k-\alpha_j}$ for $l=1,\cdots, m$. We obtain
$$
e_{i,j}(r)=\prod_{l=1}^m e_{i,j}(s_lt_l). 
$$
But for every $l$, we have
$$
e_{i,j}(s_lt_l)=[e_{i,k}(s_l), e_{k,j}(t_l)],
$$
where $e_{i,k}(s_l)$ and  $e_{k,j}(t_l)$ are in $\E_n(A)(\alpha_1,\cdots,\alpha_n)$.
It follows immediately that 
$$e_{i,j}(r)\in [\E_n(A)(\alpha_1,\cdots,\alpha_n), \E_n(A)(\alpha_1,\cdots,\alpha_n)].$$ 
Hence 
$$
\E_n(A)(\alpha_1,\cdots,\alpha_n) =[\E_n(A)(\alpha_1,\cdots,\alpha_n), \E_n(A)(\alpha_1,\cdots,\alpha_n)].
$$
This finishes the proof.
\end{proof}
On the other hand, a stable elementary subgroup is always  perfect and normal in its corresponding stable general linear group.
\begin{Lem}\label{perfect-2}
$\E(A)(\alpha_1,\cdots,\alpha_n)$ is a perfect group.
\end{Lem}
\begin{proof}
We prove the lemma by showing that the image of $\E_n(A)(\alpha_1,\cdots,\alpha_n)$ under the stablizing homomorphism is a subgroup of  the commutator subgroup $$[\E_{2n}(A)(\alpha_1,\cdots,\alpha_n)^2,\E_{2n}(A)(\alpha_1,\cdots,\alpha_n)^2].$$ 

Let  $e_{i,j}(r)$ be a generator of $\E_n(A)(\alpha_1,\cdots,\alpha_n)$.   We have the identity
$$
e_{i,j}(r)=[e_{i,n+j}(r), e_{n+j,j}(1)]
$$
where $e_{i,n+j}(r)$ and $e_{n+j,j}(1)$  belong to $\E_{2n}(A)(\alpha_1,\cdots,\alpha_n)^2$. Hence
$$e\in[\E_{2n}(A)(\alpha_1,\cdots,\alpha_n)^2,\E_{2n}(A)(\alpha_1,\cdots,\alpha_n)^2].$$
Therefore,
$$
\E(A)(\alpha_1,\cdots,\alpha_n)=[\E(A)(\alpha_1,\cdots,\alpha_n),\E(A)(\alpha_1,\cdots,\alpha_n)],
$$
which proves the lemma.
\end{proof}

\begin{Lem}\label{[gg]=e} For any $h\in \GL_{n}(A)(\alpha_1,\cdots,\alpha_n)$, we have
$$
\left(\begin{array}{cc}h&0\\0&h^{-1}\end{array}\right)
\in \E_{2n}(A)(\alpha_1,\cdots,\alpha_n)^2.$$ 
\end{Lem}
\begin{proof}
We have the matrix equation
$$
\left(\begin{array}{cc}h&0\\0&h^{-1}\end{array}\right)=\left(\begin{array}{cc}I_n&h\\0&I_n\end{array}\right)
\left(\begin{array}{cc}I_n&0\\-h^{-1}&I_n\end{array}\right)\left(\begin{array}{cc}I_n&h\\0&I_n\end{array}\right)\left(\begin{array}{cc}0&-I_n\\I_n&0\end{array}\right).
$$
By a standard argument, for example that in \cite[Corollary 2.1.3]{Rosen}, all  four factors in the above equation lie in $\E_{2n}(A)(\alpha_1,\cdots,\alpha_n)^2$. 
\end{proof}
\begin{The}[Graded Whitehead Lemma]\label{White}
Let $A$ be a $\Gamma$ graded ring. Then 
$$[\GL(A)(\alpha_1,\cdots,\alpha_n),\GL(A)(\alpha_1,\cdots,\alpha_n)]= \E(A)(\alpha_1,\cdots,\alpha_n).$$
Therefore $\E(A)(\alpha_1,\cdots,\alpha_n)$ is a normal subgroup of $\GL(A)(\alpha_1,\cdots,\alpha_n)$.
\end{The}
\begin{proof}
Suppose that   $g$ and $h$ are two elements in  $\GL(A)(\alpha_1,\ldots,\alpha_n)$. Then there exists an $m\in \mathbb Z$ such that
$g,h\in \GL_{nm}(A)(\alpha_1,\cdots,\alpha_n)^m$.
Consider the images   of $g$ and $h$ in $\GL_{2nm}(A)(\alpha_1,\cdots,\alpha_n)^{2m}$ , which will be denoted by $g'$ and $h'$, respectively. 
The matrices  $g'$ and $h'$  take the following forms
$$
g'=\left(\begin{array}{cc}
g&0\\
0&I_{nm}
\end{array}\right) \quad\text{and} \quad h'=\left(\begin{array}{cc}
h&0\\
0&I_{nm}
\end{array}\right).
$$
Now apply the identity
$$
\left(\begin{array}{cc}ghg^{-1}h^{-1}&0\\0&1\end{array}\right)=\left(\begin{array}{cc}gh&0\\0&h^{-1}g^{-1}\end{array}\right)
\left(\begin{array}{cc}g^{-1}&0\\0&g\end{array}\right)\left(\begin{array}{cc}h^{-1}&0\\0&h\end{array}\right).
$$
By Lemma~\ref{[gg]=e}, all of the factors on the right lie in $\E_{2nm}(A)(\alpha_1,\ldots,\alpha_n)^{2m}$. This shows that 
$$[g,h]\in \E(A)(\alpha_1,\cdots,\alpha_n).$$
Hence
$$[\GL(A)(\alpha_1,\cdots,\alpha_n),\GL(A)(\alpha_1,\cdots,\alpha_n)]= \E(A)(\alpha_1,\cdots,\alpha_n).$$
It follows that $\E(A)(\alpha_1,\cdots,\alpha_n)$ is a normal subgroup of $\GL(A)(\alpha_1,\cdots,\alpha_n)$. This finishes the proof.
\end{proof}
Theorem~\ref{White} amounts to saying that the quotient group  $\GL(A)(\alpha_1,\cdots,\alpha_n)/\E(A)(\alpha_1,\cdots,\alpha_n)$ is an abelian group for every finite family  $\{\alpha_1,\cdots,\alpha_n\}$. 
\begin{Lem}\label{lem:4.7}
Let $S$ be a finite family of elements of $\Gamma$. Then 
$$
K_1(A)(S)\simeq \GL(A)(S)/\E(A)(S).
$$
\end{Lem}
The lemma is an easy generalization  of the corresponding result in classical $K$-theory. We skip its proof and refer the interested reader to \cite[Theorem~5.4]{Lam} or \cite[Theorem~3.1.7]{Rosen}.

It follows by Lemma~\ref{lem:4.7} that for any $\Gamma$ graded ring $A$,  $K_1^{gr}(A)$ is a direct limit of matrix  rings, i.e.,
$$
K_1^{gr}(A)\simeq \varinjlim_{S}K_1(A)(S)\simeq\varinjlim_{S} \GL(A)(S)/\E(A)(S) .
$$
This gives  ``locally'' a matrix description for $K_1^{gr}(A)$.


\section{Relative $K_1^{gr}$ and an exact sequence}\label{s:5}
An important property of $K$-theory groups is the existence of  exact sequences relating the $K_i(A)$, $K_i(A/I)$ and the relative $K$-group $K_i(A,I)$ for each $i\in \mathbb Z$, where $A$ is an ungraded ring containing $I$ as a two-sided ideal. In the present section, we establish such exact sequence for $K_1^{gr}$ of graded rings.

\begin{Def}
Let $A$ be a $\Gamma$ graded ring and $I$ a graded two-sided ideal of $A$. The {\em double of $A$ along $I$} is the subring of the Cartesian product $A\times A$ given by
$$
D(A,I)=\{(x,y)\in A\times A\mid x-y\in I\}.
$$
Its grading is induced from the grading of $A\times A$.
\end{Def}
We denote the projection onto the $i$th coordinate of $D(A,I)$ by $p_i$. Then there is a  split exact sequence of graded rings
$$
\xymatrix@1{0\ar[r]& I\ar[r]& D(A,I)\ar[r]^{\hspace{15pt} p_1}& A\ar[r]&0.}
$$

\begin{Def}
The {\em relative $K_{1}^{gr}$ group}  of a graded ring $A$ and a graded ideal $I$ is defined by
$$
K_1^{gr}(A,I)=\ker \Big( (p_1)_*: K_1^{gr}(D(A,I))\to K_1^{gr}(A)\Big).
$$
\end{Def}
As in the case of ungraded rings, we need a more computable construction of   $K_{1}^{gr}(A,I)$. This will be done locally.
\begin{Def}
Let $S$ be a finite family of elements  of $\Gamma$. We define
$$
K_1(A,I)(S)=\ker \Big( (p_1)_*: K_1(D(A,I))(S)\to K_1(A)(S)\Big).
$$
\end{Def}

\begin{The}\label{relative-White}
Let $A$ be a $\Gamma$ graded ring with graded two-sided ideal $I$. Then $\E(A,I)(\alpha_1,\cdots,\alpha_n)$ is normal in $\GL(A,I)(\alpha_1,\cdots,\alpha_n)$, and
$$
\GL(A,I)(\alpha_1,\cdots,\alpha_n)/\E(A,I)(\alpha_1,\cdots,\alpha_n)\cong K_1(A,I)(\alpha_1,\cdots,\alpha_n),
$$
for every finite family $\{\alpha_1,\cdots,\alpha_n\}$ of elements of $\Gamma$.
\end{The}
\begin{proof}
Suppose
$e\in \E(A,I)(\alpha_1,\cdots, \alpha_n)$ and $g\in \GL(A)(\alpha_1,\cdots, \alpha_n)$. Then there exists an $m$, such that $e\in \E_{nm}(A)(\alpha_1,\cdots, \alpha_n)^m$ and  $g\in \GL_{nm}(A)(\alpha_1,\cdots, \alpha_n)^m$.  
Then we have the identity
$$
\left(\begin{array}{cc}
geg^{-1}&0\\0&I_{nm}
\end{array}\right)=\left(\begin{array}{cc}
g&0\\0&g^{-1}
\end{array}\right)
\left(\begin{array}{cc}
e&0\\0&I_{nm}
\end{array}\right)
\left(\begin{array}{cc}
g^{-1}&0\\0&g
\end{array}\right).
$$
By Lemma~\ref{[gg]=e}, the first and third factors on the right are in $\E_{2nm}(A)(\alpha_1,\cdots, \alpha_n)^{2m}$. This proves the first assertion.

Suppose $(g_1,g_2)\in  \GL(D(A,I))(\alpha_1,\cdots,\alpha_n)$ and maps to the identity of $K_1^{gr}(A,I)$ under $(p_1)_*$.  This means that $g_1\in \E(A)(\alpha_1,\cdots,\alpha_n)$. 
Thus there is an $m\in\mathbb N$ such that 
$
g_1,g_2\in \GL_{nm}(A)(\alpha_1,\cdots, \alpha_n)^m.
$
Thus $g_1$ can be write as a finite product of elementary matrices, namely we have $g_1=\prod_k e_{i_k,j_k}(a_k)$. But
$$
(g_1,g_1)=\prod_k e_{i_k,j_k}(a_k,a_k)\in \E(D(A,I)(\alpha_1,\cdots, \alpha_n)^m).
$$
Multiplying $(g_1,g_2)$ by $(g_1,g_1)^{-1}$ does not change the class of $(g_1,g_2)$ in $K_1(D(A,I))(\alpha_1,\cdots,\alpha_n)$, and 
$$
(g_1,g_2)(g_1, g_1)^{-1}=(1, g_2g_1^{-1}).
$$
Since $(1, g_2g_1^{-1})\in \GL(D(A,I))(\alpha_1,\cdots,\alpha_n)$, we deduce that $g_2g_1^{-1}\in \GL(A,I)(\alpha_1,\cdots,\alpha_n)$. 

Conversely every $g\in \GL(A,I)(\alpha_1,\cdots,\alpha_n)$ corresponds to $(1,g)$ in $\GL(D(A,I))(\alpha_1,\cdots,\alpha_n)$, which defines an element in $K_1(D(A,I))(\alpha_1,\cdots,\alpha_n)$. 

To show that $\GL(A,I)(\alpha_1,\cdots,\alpha_n)/\E(A,I)(\alpha_1,\cdots,\alpha_n)\cong K_1(A,I)(\alpha_1,\cdots,\alpha_n)$, we need only to  check that $(1, g)\in \E(D(A,I))(\alpha_1,\cdots,\alpha_n)$ if and only if $ g\in \E(A,I)(\alpha_1,\cdots,\alpha_n)$ for any $g\in \GL(A,I)(\alpha_1,\cdots,\alpha_n)$.

For one direction, suppose that $g\in \E(A,I)(\alpha_1,\cdots,\alpha_n)$. Then there exists an $m\in\mathbb N$
such that $g\in \E_{nm}(A,I)(\alpha_1,\cdots,\alpha_n)^m$.  Then $g=\prod_k S_ke_{i_k,j_k}(a_k)S_k^{-1}$ where $S_k\in \E_{nm}(A)(\alpha_1,\alpha_2,\cdots,\alpha_n)^m$ and $a_k\in I$. It suffices to show that $(1, S_ke_{i_k,j_k}(a_k)S_k^{-1})\in \E_{nm}(D(A,I))(\alpha_1,\alpha_2,\cdots,\alpha_n)^m$ for each $k$. But
$$
(1, S_ke_{i_k,j_k}(a_k)S_k^{-1})=(S_k,S_k)e_{i,j}(0,a_k)(S_k^{-1},S_k^{-1})
$$
and all the three factors on the right lie in $\E(D(A,I))(\alpha_1,\alpha_2,\cdots,\alpha_n)$.

For the other direction, suppose that $(1, g)\in \E(D(A,I))(\alpha_1,\cdots,\alpha_n)$. There exists an $m\in\mathbb Z$ such that $(1,g)\in \E_{nm}(A,I)(\alpha_1,\cdots,\alpha_n)^m)$. Thus
$$
(1,g) =\prod_{k=1}^re_{i_k,j_k}(a_k,b_k)\in \E_{nm}(A,I)(\alpha_1,\cdots,\alpha_n)^m,
$$
where $\prod_{k=1}^re_{i_k,j_k}(a_k)=1$. Note for each $k$,
$$
e_{i_k,j_k}(a_k,b_k)=e_{i_k,j_k}(a_k,a_k)e_{i_k,j_k}(0,b_k-a_k)=(S_k,S_k)(1, T_{k}),
$$
where
$$
S_k=e_{i_k,j_k}(a_k), \quad T_k=e_{i_k,j_k}(b_k-a_k),  \quad b_k-a_k\in I.
$$
Thus we have
\begin{eqnarray*}
(1, g)&=&\prod_{k=1}^re_{i_k,j_k}(a_k,b_k)\\
&=&\prod_{k=1}^r (S_k, S_kT_k)\\
&=&(S_1,S_1T_1S_1^{-1})(S_2,S_1S_2T_1S_2^{-1}S_1^{-1})\\
&&\qquad \cdots (S_2,S_1S_2\cdots S_rT_1)\\
&=&\Big(I_n, (S_1T_1S_1^{-1})(S_1S_2T_1S_2^{-1}S_1^{-1})\\
&&\qquad\cdots(S_1S_2\cdots S_rT_1 S_r^{-1}\cdots S_2^{-1}S_1^{-1})
\Big)
\end{eqnarray*}
since $S_1S_2\cdots S_r=1$, and we have written $g$ as a product of generators of $E_n(A,I)(\alpha_1,\dots,\alpha_n)$. This finishes the proof.
\end{proof}

\begin{The}\label{Exact}
Let $A$ be a graded ring with a graded two-sided ideal $I$ and $S$ a finite family of elements of $\Gamma$. Then there is a natural exact sequence of groups
$$
\xymatrix@1{K_1(A,I)(S)\ar[r]^{\quad(p_2)_*}& K_1(A)(S)\ar[r]^{q_*}& K_1(A/I)(S)},
$$
where $(p_2)_*$ is induced by the projection $p_2:D(A,I)\to A$ and $q$ is induced by the quotient map $A\to A/I$.
\end{The}
\begin{proof}
We have seen that any class in $K_1(A,I)(S)$ is represented by 
$$(1,g)\in \GL(D(A,I))(S)$$
with $g\in \GL(A,I)(S)$. We obtain immediately that $q_*\circ (p_2)_*([g])=1$ in $K_1(A/I)(S)$.

Conversely, if $g\in K_1(A)(S)$ and $q_*(g)=1$ in $K_1(A/I)(S)$, then  $g$ is represented by 
$g'\in \GL(A)(S)$.
$q(g')\in \E(A/I)(S)$. Since $\E(A)(S)$ maps surjectively  onto $\E(A/I)(S)$, we may find $h\in \E(A)$ such that $q(g')=q(h)$. Therefore $q(gh^{-1})=1\in \GL(A/I)(S)$. Now $(1, gh^{-1})\in \GL(D(A,I))(S)$ and we  finish the proof by Theorem~\ref{relative-White}. \end{proof}

Since taking direct limits preserve exactness, we obtain from  Theorem~\ref{Exact} the usual $K$-theory sequence for  graded $K_1^{gr}$:
\begin{The}\label{Exact2}
Let $A$ be a graded ring with a graded two-sided ideal $I$. Then there is a natural exact sequence of $\mathbb Z[\Gamma]$-modules
$$
\xymatrix@1{K_1^{gr}(A,I)\ar[r]^{\quad(p_2)_*}& K_1^{gr}(A)\ar[r]^{q_*}& K_1^{gr}(A/I)},
$$
where $(p_2)_*$ is induced by the projection $p_2:D(A,I)\to A$ and $q$ is induced by the quotient map $A\to A/I$.
\end{The}
\begin{Rem}
As in classical $K$-theory, the sequence above can be extended to graded $K_0$ groups
$$
\xymatrix@1{K_1^{gr}(A,I)\ar[r]^{\quad(p_2)_*}& K_1^{gr}(A)\ar[r]^{q_*}& K_1^{gr}(A/I)\ar[r]&K_0^{gr}(A,I)\ar[r]& K_0^{gr}(A)\ar[r]& K_0^{gr}(A/I)}.
$$
 It is a straightforward adoption of the proof for the classical result, see \cite{Bass} or \cite{Rosen} and see \cite{RH2} for the exactness of  the extended sequence in classical $K$-theory.
\end{Rem} 

\section{$K_1^{gr}$ of cross products}\label{s:final}
It is known that for any $\Gamma$ strongly graded  ring $A$, the category $\mM A$ of graded $A$-modules  is equivalent to the category $\mM {A_0}$ of $A_0$-modules  \cite[Theorem~2.8]{Dade}. Therefore $K_i^{gr}(A)$ is isomorphic to $K_i(A_0)$ as a group \cite[Section~3.8]{RH2}.  In this section, we will give an elementary proof of the fact  $K_1^{gr}(A)\cong K_1(A_0)$ as groups when $A$ is strongly graded, and then show that the action of $\Gamma$ on $K_1^{gr}(A)$ is always trivial, when $A$ is a cross product.

\begin{Lem}
Let $A$ be a $\Gamma$ strongly graded ring. For any free graded module $A(\alpha)$ of rank $1$, there exist a finitely generated graded projective module $P$ and a free graded module $A(0)^m$ of rank $m$ such that 
$$
A(\alpha)\oplus P \cong _{gr} A(0)^m.
$$
\end{Lem}
\begin{proof}
By definition, $A(\alpha)$ is an $\alpha$-suspended module of $A$. Since $A(\alpha)$ is $\alpha$-suspended,  the identity $1$ of $A$ is of level $-\alpha$ in $A(\alpha)$. To make things clear, we denote the $\alpha$-suspended identity $1$ by $1(\alpha)$. 

Since $A$ is strongly graded, we can find elements $s_1,\ldots, s_m\in \A_{-\alpha}$ and $t_1, \ldots, t_m\in A_{\alpha}$ such that 
$$
\sum_{i=1}^m s_it_i=1.
$$
Since all $t_i$'s are  in $A_{\alpha}$, it follows immediately that the $t_i1(\alpha)$ 's are in $A(\alpha)_0$. Furthermore, the equation
$$
\sum_{i=1}^m s_it_i1(\alpha)=1(\alpha)
$$
implies that $t_11(\alpha),\ldots,t_m1(\alpha)$ generated $A(\alpha)$ as graded module. 

The free graded module $A(0)^m$ has a basis $(a_1, \ldots,a_m)$ with $a_1,\ldots,a_m\in A(0)_0$. For each $i$, assign  $a_i$ in $A(0)^m$ to $ t_i1(\alpha)$ in $A(\alpha)$. Extending the assignment linearly, we get a graded homomorphism $\varphi: A(0)^m\to A(\alpha)$, which is  surjective. The graded homomorphism $\varphi$ induces
 an exact sequence of graded modules
$$
\xymatrix@1{0\ar[r]& \ker(\varphi)\ar[r]& A(0)^m\ar[r]^{\varphi}&A(\alpha)\ar[r]&0 },
$$
which is split by Proposition~\ref{gr-projective}. We obtain that
$$A(\alpha)\oplus \ker(\varphi) \cong _{gr} A(0)^m,
$$
and   the proof is finished.
\end{proof}
A direct consequence of the above lemma is the following.

\begin{Lem}
Let $A$ be a strongly graded $\Gamma$ graded ring. Then
every free graded $A$-module of finite rank is a direct summand of a free graded module $A(0,\ldots,0)$ of finite rank. Furthermore, every finitely generated graded projective $A$-module is a direct summand of a free graded module $A(0,\ldots,0)$ of finite rank.
\end{Lem}
 This phenomenon reduces $K_1^{gr}(A)$ to $K_1(A)(0)$ which is $\GL(A)(0)/\E(A)(0)$ by Lemma~\ref{lem:4.7}. By the definition of the $K_1$ of ungraded rings, $\GL(A)(0)/\E(A)(0) \cong K_1(A_0)$. Hence we have
\begin{Lem}
Let $A$ be a strongly graded $\Gamma$ graded ring. Then
$$
K_1^{gr}(A)\cong K_1(A_0)
$$
as groups.
\end{Lem}
In  case  $A$ is a cross product, the action of $\Gamma$ is always trivial.
 
\begin{Lem}\label{lem:7.1}
Let $\{\alpha_1,\cdots,\alpha_n\}$ be a finite family of elements of  $\Gamma$ and $\lambda$ an element in $\Gamma$. For any invertible homogeneous element $r\in A_\lambda$, we have 
$$
\left(\begin{array}{cc}
0&-r^{-1}I_n\\rI_n&0
\end{array}\right)\in \E_{2n}(A)(\alpha_1,\cdots,\alpha_n,\lambda+\alpha_1,\cdots,\lambda+\alpha_n).
$$
\end{Lem}
\begin{proof}
The assertion directly follows from the matrix identity
$$
\left(\begin{array}{cc}
0&-r^{-1}I_n\\rI_n&0
\end{array}\right)=\left(\begin{array}{cc}
I_n&-r^{-1}I_n\\0&I_n
\end{array}\right)\left(\begin{array}{cc}
I_n&0\\rI_n&I_n
\end{array}\right)
\left(\begin{array}{cc}
I_n&-r^{-1}I_n\\0&I_n
\end{array}\right)
$$
where all three factors on the right  are in $\E(A(\alpha_1,\cdots,\alpha_n,\lambda+\alpha_1,\cdots,\lambda+\alpha_n))$.
\end{proof}
\begin{Lem}\label{lem:7.2}
For any $h\in \E_n(A)(\alpha_1,\cdots,\alpha_n)$ and any invertible homogeneous element $r\in A_\lambda$, we have
$$
\left(\begin{array}{cc}
h&0\\0&rh^{-1}r^{-1}
\end{array}\right)\in\E_{2n}(A)(\alpha_1,\cdots,\alpha_n,\lambda+\alpha_1,\cdots,\lambda+\alpha_n).
$$
\end{Lem}
\begin{proof}
We have the matrix equation
$$
\left(\begin{array}{cc}
h&0\\0&rh^{-1}r^{-1}
\end{array}\right)=\left(\begin{array}{cc}
I_n&hr^{-1}\\0&I_n
\end{array}\right)
\left(\begin{array}{cc}
I_n&0\\-rh^{-1}&I_n
\end{array}\right)
\left(\begin{array}{cc}
I_n&hr^{-1}\\0&I_n
\end{array}\right)
\left(\begin{array}{cc}
0&-r^{-1}I_n\\rI_n&0
\end{array}\right)
$$
where the first three factors are in $\E_{2n}(A)(\alpha_1,\cdots,\alpha_n,\lambda+\alpha_1,\cdots,\lambda+\alpha_n)$ and the last factor is also there by Lemma~\ref{lem:7.1}.
\end{proof}
\begin{The}\label{cross}
Suppose the $\Gamma$ graded ring $A$ is a cross product. Then the action of $\Gamma$ on $K_1^{gr}(A)$ is trivial. 
\end{The}
\begin{proof}
Let $[(P,\varphi)]$ be an element in $K_1^{gr}(A)$. Then there is a finite family $\{\alpha_1,\cdots,\alpha_n\}$ of elements of $\Gamma$ such that $[(P,\varphi)]\in K_1^{gr}(A)(\alpha_1,\cdots,\alpha_n)$. Take $h$ to be a matrix which represents $[(P,\varphi)]$ in 
$K_1^{gr}(A)(\alpha_1,\cdots,\alpha_n)$. For any $\gamma\in \Gamma$, the result of $\gamma$ acting  on $[P,\varphi]$ is $[(P(\gamma),\mS_\gamma(\varphi))]$, which lies in $K_1^{gr}(A)(\lambda+\alpha_1,\cdots,\lambda+\alpha_n)$.  It is easy to see  that $[(P(\gamma),\mS_\gamma(\varphi))]$ is represented by $g=r hr^{-1}$ in $K_1^{gr}(A)(\lambda+\alpha_1,\cdots,\lambda+\alpha_n)$ with $r$ an invertible element of degree $\lambda$. The images of $h$ and $g^{-1}$ in the larger group
$K_1^{gr}(A)(\alpha_1,\cdots,\alpha_n,\lambda+\alpha_1,\cdots,\lambda+\alpha_n)$ are 
$$h'=\left(\begin{array}{cc}
h&0\\0&I_n
\end{array}\right)\qquad\text{ and }\qquad g'=\left(\begin{array}{cc}
I_n&0\\0&rh^{-1}r^{-1}
\end{array}\right),$$ respectively. By Lemma~\ref{lem:7.2}, the product of $h'$ and $g'$ is zero in $K_1^{gr}(A)(\alpha_1,\cdots,\alpha_n,\lambda+\alpha_1,\cdots,\lambda+\alpha_n)$. This shows that $g=h$ in $K_1^{gr}(A)$. Therefore the action of $\Gamma$ is trivial.
\end{proof}
\begin{Rem}
Theorem~\ref{cross} can also be concluded using the properties of cross product directly. See \cite[Example 3.1.9]{RH2}.
\end{Rem}
\noindent{\bf Acknowledgement.}
I would like to thank Roozbeh Hazrat  for very productive discussions and for sharing me generously his  unpublished  manuscript \cite{RH2} and to Prof.~Anthony Bak for constructive suggestions. I would like also to thank anonymous referee for very careful reading the manuscript and detailed comments. 

\end{document}